\documentclass[
final
]{dmtcs-episciences}

\usepackage[utf8]{inputenc}
\usepackage{subfigure}

\usepackage{graphicx, amsmath,amsthm,amssymb, multirow}
\usepackage{url}
\newtheorem{theorem}{Theorem}[section]
\newtheorem{lemma}[theorem]{Lemma}
\newtheorem{proposition}[theorem]{Proposition}
\newtheorem{corollary}[theorem]{Corollary}

\theoremstyle{definition}

\newtheorem{example}[theorem]{Example}

\usepackage[round]{natbib}

\hypersetup{hidelinks}
\author[Kassie Archer and Noel Bourne]{Kassie Archer\affiliationmark{1}\thanks{ORCID: 0000-0001-5702-8723}
  \and Noel Bourne\affiliationmark{2}}
\title[Pattern avoidance in compositions and powers of permutations]{Pattern avoidance in compositions and powers of permutations}
\affiliation{
  Department of Mathematics, United States Naval Academy, Annapolis, USA\\
  Department of Mathematical Sciences, Tennessee State University, Nashville, USA}
\keywords{pattern avoidance, chain avoidance, generating functions, generalized Fibonacci numbers}

\renewcommand{\S}{\mathcal{S}}
\renewcommand{\d}{\mathbf{d}}
\renewcommand{\c}{\mathbf{c}}

\begin{document}
\publicationdata{vol. 28:1, Permutation Patterns 2025}{2026}{4}{10.46298/dmtcs.17199}{2025-12-28; 2025-12-28; 2026-05-04}{2026-05-04}
\maketitle
\begin{abstract}
     A permutation $\pi$ is said to avoid a chain $(\sigma:\tau)$ of patterns if $\pi$ avoids $\sigma$ and $\pi^2$ avoids $\tau.$ In this paper, we define a notion of pattern avoidance in compositions of positive integers and use that idea to enumerate permutations of length $n$ that avoid the chain $(312,321:\sigma)$ for any pattern $\sigma\in \bigcup_{m\geq 1} \mathcal{S}_m$. We also enumerate those permutations that avoid the chain $(312,4321:\sigma)$ for any $\sigma\in\mathcal{S}_3.$ 
\end{abstract}

\section{Introduction}

Let $\S_n$ denote the set of permutations of $[n]=\{1,2,\ldots, n\}.$ In this paper, we always write permutations $\pi\in\S_n$ in their one-line form as $\pi=\pi_1\pi_2\ldots\pi_n.$ We say that a permutation $\pi\in\S_n$ \textit{contains} a pattern $\sigma\in\S_k$ if there is some subsequence of $\pi$ that appears in the same relative order as the elements of $\sigma$. We say that $\pi$ \textit{avoids} $\sigma$ if it does not contain it. For example, the permutation $\pi = 416728359$ contains $312$ since $635$ is a subsequence of $\pi$ that appears in the same relative order as 312. However, it avoids 321 since there is no subsequence of $\pi$ of length 3 that appears in the same order as 321. We denote by $\S_n(\sigma)$ the set of permutations in $\S_n$ that avoid $\sigma$. The study of pattern-avoiding permutations began in \cite{Knuth,SS85}; see also \cite{Bona}.

Recently, the notion of strong avoidance was defined in \cite{BS19}, in which a permutation $\pi$ \textit{strongly avoids} a pattern $\sigma$ if both $\pi$ and $\pi^2$ avoid $\sigma$. The authors of that paper consider permutations that strongly avoid some $\sigma\in\S_3.$ They completely enumerate those that strongly avoid 312 or 123, and find bounds on the number of those that strongly avoid 321 or 132. Later in \cite{BD20}, permutations that strongly avoid the sets $\{132, 3421\}$, $\{213, 4312\}$, and $\{321, 3412\}$ were enumerated and in \cite{P23}, those that strongly avoid $\{321, 1342\}$ were shown to be enumerated by $2F_{n+1}-n-2$, where $F_{n+1}$ is the $(n+1)$-th Fibonacci number. In \cite{PG24}, the authors considered those permutations that avoid certain patterns of length 3 or 4 and strongly avoid the pattern 312.

The notion of strong avoidance was generalized in \cite{AG24} to the notion of chain avoidance, in which $\pi$ and $\pi^2$ may avoid different patterns. We say $\pi$ \textit{avoids the chain} $(\sigma:\tau)$ if $\pi$ avoids $\sigma$ and $\pi^2$ avoids $\tau.$ As with regular pattern avoidance, we can also consider the case where $\pi$ and $\pi^2$ avoid sets of patterns. For example, we say $\pi$ avoids $(312,4321: 213)$ if $\pi$ avoids both $312$ and $4321$ while $\pi^2$ avoids 213. In \cite{AG24}, the authors enumerate permutations that avoid the chain $(312,213:\sigma)$ or $(231,321:\sigma)$ for any $\sigma\in\S_3$.

In this paper, we will enumerate those permutations that avoid the chain $(312,321:\sigma)$ for any pattern $\sigma$ of any length. First, in Section~\ref{sec:compositions} we define a notion of pattern avoidance in compositions of $n$. We note that this definition differs from the definition of pattern avoidance for compositions that appears in \cite{HM06,SW06}, but is related to the notion of composition containment in \cite{SV,V} and of dominating compositions in \cite{JMS, MS22}. In Section~\ref{sec:312-321}, we then define a set of compositions associated to any possible pattern in $\pi^2$ for a given $\pi$ that avoids $312$ and 321. We use this set of compositions and results from Section~\ref{sec:compositions} to find recursive formulas for permutations avoiding 312 and 321 whose square avoids any given pattern $\sigma$. Finally, in Section~\ref{sec:312-4321}, we enumerate those permutations avoiding the chain $(312,4321:\sigma)$ for any $\sigma\in \S_3$ using different methods.

Throughout this paper, we will use the Fibonacci numbers  (OEIS A000045) and their generalizations. We define the Fibonacci sequence recursively as $F_{n}=F_{n-1}+F_{n-2}$ with $F_0=0$ and $F_1=1.$ For $k\geq 2$, we similarly define the $k$-nacci sequence to be \[F^{(k)}_n = \sum_{i=1}^k F_{n-i}^{(k)}\] with $F^{(k)}_i=0$ for $0\leq i<k-1$ and $F^{(k)}_{k-1}=1$. In the special case where $k=3$, we call these the Tribonacci numbers (OEIS A000073) and denote them by $T_n$, and in the case where $k=4$, we call these the Tetranacci numbers (OEIS A000078) and denote them by $Q_n.$

\section{Composition avoidance}\label{sec:compositions}
A composition $\d$ of $n$ is a sequence of positive integers $\d=(d_1,d_2,\ldots, d_m)$ so that \[d_1+d_2+\cdots+d_m=n.\]
Given a composition $\d = (d_1,d_2,\ldots,d_m)$ of $n$, we say that the composition $\d$ \textit{contains} another composition $\c= (c_1,c_2,\ldots,c_k)$ if there is some $i_1<\cdots < i_k$ so that $d_{i_j}\geq c_j$ for all $1\leq j \leq k$. In this case, we say the sub-composition $(d_{i_1},\ldots,d_{i_k})$ \textit{dominates} $\c$. (The notion of one composition dominating another also appears when studying pattern avoidance in set partitions as in \cite{JMS, MS22}. There, the term ``dominate'' corresponds to our notion of containment.) We say that a composition $\d$ \textit{avoids} $\c$ if $\d$ does not contain $\c$. 
For example, the composition 
$\d = (3,6,4,2,5,1)$ contains $(4,1,2,3)$ since $(6,4,2,5)$ dominates $(4,1,2,3)$ and $\d$ avoids $(4,1,3,2)$ since there is no subsequence of $\d$ of length 4 that dominates this pattern.
As another example, $\d=(2,1,2,1,1,2)$ avoids the composition $(3)$ since there is no element of $\d$ greater than or equal to 3. 
Note that we are not requiring  $(d_{i_1},\ldots,d_{i_k})$ and $\c=(c_1,\ldots, c_k)$ to be order-isomorphic as we do for permutation pattern containment. Let $b_n(\c)$ denote the number of compositions of $n$ that avoid the composition $\c.$ We will take $b_n(())=0$, i.e., by convention, we will say that no composition of $n$ avoids the empty composition.

Given a composition $\c = (c_1,c_2,\ldots,c_k)$ with $k\geq 1$, let $\hat{\c}$ be the composition $\hat{\c} = (c_2,\ldots,c_k)$. Note that if $k=1$, then $\c=(c)$ for some $c$, so $\hat\c = ()$, the empty composition.

\begin{theorem}\label{thm:one comp}
    For $n\geq 1,$ and any composition $\c = (c_1,c_2,\ldots, c_k)$, then:
    \[
    b_n(\c) =  \sum_{i=1}^{c_1-1} b_{n-i}(\c) + \sum_{i=c_1}^n b_{n-i}(\hat{\c})
    \]
    with $b_n(())=0$ for all $n\geq 0$, and $b_0(\c) = 1$ for all nonempty compositions $\c$.
\end{theorem}
\begin{proof}
    First, let us consider the case where $k=1$. Then $\c=(c)$ for some positive integer $c.$ Clearly, compositions that avoid this pattern are exactly those composed of elements less than $c.$ Thus $b_n(\c) = \sum_{i=1}^{c-1} b_{n-i}(\c).$ Since $b_n(())=0$ for any $n$, the theorem holds in this case. 

    Now consider the case when $k>1.$ In this case, any composition $\d=(d_1,\ldots, d_m)$ that avoids $\c$ must either have the property that $d_1<c_1$ and $(d_2,\ldots, d_m)$ avoids $\c$ or that $d_1\geq c_1$ and $(d_2,\ldots, d_m)$ avoids $\hat\c$. The statement of the theorem follows immediately from this observation.  
\end{proof}

\begin{example}
    Let us consider the example with $\c=(3,2)$. Then $c_1=3$ and $\hat\c=(2)$. By Theorem~\ref{thm:one comp}, we have that \[
    b_n((3,2)) = \sum_{i=1}^2 b_{n-i}((3,2)) + \sum_{i=3}^n b_{n-i}((2)).
    \] Since for any $n,$ we have $b_n((2))=1$ (since any composition avoiding $(2)$ must only be composed of $1$'s), if we take $b_n:=b_n((3,2))$, then $b_n$ satisfies 
    \[b_n=b_{n-1}+b_{n-2}+n-2\] with initial conditions $b_1=1$ and $b_2=2$. Solving this recurrence gives us that $b_n((3,2))=F_{n+3}-n-1.$
\end{example}

\begin{example}\label{ex:Fibonacci}
Let $\c=(k+1)$ for some $k\geq1$. Note that compositions avoiding $\c$ are exactly the compositions with no part greater than $k$. The recurrence relation given by Theorem~\ref{thm:one comp} is \begin{align*}
    b_{n}((k+1))=\sum_{i=1}^{k}b_{n-i}((k+1))+\sum_{i=k+1}^{n}b_{n-i}(()).
\end{align*} Since the second sum is equal to 0, we have \begin{align*}
    b_{n}((k+1))=\sum_{i=1}^{k}b_{n-i}((k+1))
\end{align*}
with $b_0((k+1))=1$ and $b_j((k+1))=0$ when $j<0.$
Thus we must have
\begin{align*}
    b_{n}((k+1))=F_{n+k-1}^{(k)}
\end{align*}
for $n\geq 0$, where $F_{n+k-1}^{(k)}$ is the $(n+k-1)$st $k$-nacci number.
\end{example}

We can also enumerate those compositions of $n$ that avoid a set of compositions. 
Given a set of compositions $C = \{\c^{(1)}, \c^{(2)}, \ldots, \c^{(m)}\}$ with $
    \c^{(i)} = (c^{(i)}_1,c^{(i)}_2,\ldots, c^{(i)}_{k_i})$ for each $1\leq i\leq m$, ordered by their first element (so that $c^{(i)}_1\leq c^{(i+1)}_1$ for each $i$), define the set $\hat{C}^j$ to be the set of compositions
    \[\hat{C}^j = \{\hat\c^{(1)}, \hat\c^{(2)}, \ldots, \hat\c^{(j)}, \c^{(j+1)}, \ldots, \c^{(m)}\}\]
\begin{theorem}\label{thm:mult comp}
    For $n\geq 1$, $m\geq 1$, and any set of compositions $C = \{\c^{(1)}, \c^{(2)}, \ldots, \c^{(m)}\}$ with $
    \c^{(i)} = (c^{(i)}_1,c^{(i)}_2,\ldots, c^{(i)}_{k_i})$, and $c^{(i)}_1\leq c^{(i+1)}_1$ for each $i$, then:
    \[
    b_n(C) =  \sum_{i=1}^{c^{(1)}_1-1} b_{n-i}(C) + \sum_{i=c_1^{(1)}}^{c^{(2)}_1-1} b_{n-i}(\hat{C}^1)+ \cdots  + \sum_{i=c_1^{(m-1)}}^{c^{(m)}_1-1} b_{n-i}(\hat{C}^{m-1})+\sum_{i=c^{(m)}_1}^n b_{n-i}(\hat{C}^{m})
    \]
     with $b_r(C)=0$ for all $r\geq 0$ if $C$ contains the empty composition, and $b_0(C) = 1$ for all nonempty sets $C$ that do not contain the empty composition. 
\end{theorem}
\begin{proof}
    The proof follows in a similar way to that of Theorem~\ref{thm:one comp}, by considering the first element of a composition $\d$. If $\d=(d_1,d_2,\ldots,d_m)$ avoids all patterns in $C$ and $c_1^{(j)}\leq d_1<c_1^{(j+1)}$, then clearly $\d$ must avoid every pattern in $C^j$. This is because $d_1$ can act as the $c^{(i)}_1$ in any pattern $\c^{(i)}$ for $1\leq i\leq j$, but cannot act as the  $c^{(i)}_1$ in any pattern $\c^{(i)}$ for $j<i\leq m$ since the compositions $\c^{(i)}$ are ordered by their first element.
\end{proof}

\begin{example}
    Let's consider the set $C=\{(3,2),(6)\}$. Then $\hat{C}^1=\{(2),(6)\}$ and $\hat{C}^2 = \{(2),()\}$, which contains the empty composition. Notice that compositions that avoid the set $C$ are exactly those that avoid $(3,2)$ and that have no element greater than 5. Since $b_n(\hat{C}^1)=1$ and $b_n(\hat{C}^2) =0$, it is straightforward to see from Theorem~\ref{thm:mult comp} that \[b_n(C)=b_{n-1}(C)+b_{n-2}(C)+3,\] which implies that $b_n(C)=F_{n+2}+2F_n-3.$
\end{example}

Several other examples can be found in the proof of Theorem~\ref{thm:table} in the next section.

\section{Permutations avoiding the chain $(312,321: \sigma)$}\label{sec:312-321}


First, let us establish the form that permutations $\pi\in\S_n(312,321)$ can take. Let $\epsilon_1=1$, $\epsilon_2=21$, and $\epsilon_d=234\ldots d1$ for $d\geq 3$. Also, recall the definition of a direct sum of permutations. We say $\pi=\sigma\oplus\tau$ for some $\sigma\in\S_k$ and $\tau\in\S_m$ if $\pi_i=\sigma_i$ for $1\leq i\leq k$ and $\pi_i=\tau_{i-k}+k$ for $k+1\leq i\leq k+m.$ For example, $2134\oplus43215=213487659$. We can also take multiple sums. For example, if $\d=(2,5,3,1)$, then \[\bigoplus_{i=1}^4\epsilon_{d_i}=\epsilon_2\oplus\epsilon_5\oplus\epsilon_3\oplus\epsilon_1 = 21\oplus23451\oplus231\oplus1 = 21456739(10)8(11).\]
Note that the resulting permutation avoids the patterns 312 and 321. The next lemma establishes that this fact generalizes.
\begin{lemma}\label{lem:form}
    Let $n\geq 1$. A permutation $\pi\in\S_n$ avoids $312$ and $321$ if and only if $\pi = \bigoplus_{i=1}^m \epsilon_{d_i}$ for some composition $\d=(d_1,d_2,\ldots, d_m)$ of $n$.
\end{lemma}
\begin{proof}
Suppose $\pi\in\S_n(312,321)$. Since $\pi$ avoids 312, if $\pi_i=1$, then $\pi = \sigma\oplus\tau$ where $\sigma\in\S_i(312,321)$ ends in 1 and $\tau \in \S_{n-i}(312,321).$ However, since $\pi$ avoids 321, we must have $\sigma=\epsilon_i.$ Thus the lemma follows inductively.
\end{proof}

It follows immediately that there are $2^{n-1}$ permutations that avoid 312 and 321 since this set of permutations is in bijection with the set of compositions of $n$. 
Note that since squaring respects direct sums, this lemma implies that $\pi^2 = \bigoplus_{d\in \d} \epsilon^2_d$, where  $\epsilon_1^2 =1$, $\epsilon_2^2=12,$ and $\epsilon_d^2 = 34\ldots d12$ for $d\geq 3.$ 

Let $\Omega_n$ be the set of permutations $\sigma\in\S_n$ so that, when written as a direct sum of its simple parts, we have \[
 \sigma = \bigoplus_{i=1}^k \tau^{(i)}
 \] where $\tau^{(i)} \in\{1\}\cup\{23\ldots m1 : m\geq 2 \} \cup \{34\ldots m12: m \geq 3\}$ for each $i$, and let $\Omega = \bigcup_{n\geq 1}\Omega_n.$

In the next lemma, we show that we need only consider $\sigma\in \Omega$ in this section since these are exactly those patterns that can appear in $\pi^2$. Indeed, this lemma implies that if $\sigma \not\in\Omega$, then $a_n(312,321:\sigma)=2^{n-1}$ since for all $\pi \in \S_n(312,321),$ we would have that $\pi^2$ avoids $\sigma.$
\begin{lemma}
    For $\pi \in \S_n(312,321),$ $\pi^2$ can only contain patterns in $\Omega.$
\end{lemma}
\begin{proof}
    Suppose $\pi \in\S_{n}(312,321)$. By Lemma~\ref{lem:form}, $\pi = \bigoplus_{i=1}^m \epsilon_{d_i}$ for some $\d=(d_1,\ldots, d_m)$, composition of $n$, and thus $\pi^2 = \bigoplus_{i=1}^m \epsilon^2_{d_i}$. Let $\sigma$ be a pattern contained in $\pi^2$ and write $\sigma$ as a direct sum of its simple parts as 
    \[\sigma=\bigoplus_{i=1}^k \tau^{(i)}.\]

    Since our assumption is that $\tau^{(i)}$ cannot be written as a direct sum of smaller permutations, it must be the case that $\tau^{(i)}$ is a pattern contained in $\epsilon^2_{d_j}$ for some $j$. However, the possible patterns contained in $\epsilon^2_{d}$ for any $d$ are: the increasing permutation $12\ldots r$ for some $r\geq 1$, the permutation $23\ldots r1$ for some $r\geq 2$, or $34\ldots r12$ for some $r\geq 3.$ The result follows.
\end{proof}

Let us note that the number of permutations in the set $\Omega_n$ satisfies the recurrence
\[
|\Omega_n| = |\Omega_{n-1}| + \sum_{i=2}^n |\Omega_{n-i}| + \sum_{i=3}^n |\Omega_{n-i}|,
\] 
 found by considering the first element of the direct sum, with $|\Omega_0|=1$, and is thus given by the sequence A052980 from OEIS. We have that $\Omega_1=\{1\}$, $\Omega_2=\{12,21\}$, $\Omega_3=\{123,132,213,231,312\}$, and 
 \[
 \Omega_4 = \{1234, 1243, 1324, 1342, 1423, 2134, 2143, 2314, 2341, 3124, 3412\}.
 \]

 For any pattern $\sigma\in\Omega,$ we can construct a set of compositions, which we will denote $C(\sigma)$. Let us first consider the case when $\sigma = 12\ldots k = \bigoplus_{i=1}^k 1$. Then the set of compositions $C(\sigma)$ are exactly those compositions obtained by taking all compositions of $k$ and adding 2 to any element greater than or equal to 3. (There are thus $2^{k-1}$ compositions in $C(12\ldots k)$.)

 \begin{example}
     Consider $\sigma=12345$. There are 16 compositions of 5. Taking each composition and adding 2 to each element greater than or equal to 3 yields 
     \begin{align*}
        C(\sigma)  = & \, \{(1,1,1,1,1), (1,1,1,2), (1,1,2,1),(1,2,1,1), (2,1,1,1),(1,2,2),(2,1,2),(2,2,1), \\ &\, (1,1,5), (1,5,1), (5,1,1),(2,5),(5,2),(1,6),(6,1),(7)\}
     \end{align*}
     Notice that if you take $\pi = \bigoplus_{i=1}^m \epsilon_{d_i}$ for any $\d=(d_1,\ldots, d_m)\in C(12345)$, $\pi^2$ will contain a $\sigma=12345$ pattern. For example, if $\d=(5,1,1)$, then $\pi = 23451 \oplus 1\oplus 1 = 2345167$ and $\pi^2 = 34512\oplus 1\oplus 1 = 3451267$, which contains the subsequence 34567, which is a 12345 pattern.
 \end{example}
 
 Now let us consider a more general pattern $\sigma = \bigoplus_{i=1}^k \tau^{(i)}.$
 First construct a composition $\c(\sigma) = (c_1,\ldots, c_k)$
 where $c_i = |\tau^{(i)}|$ if $\tau^{(i)} \in\{1\}\cup \{34\ldots m12: m \geq 3\}$ and  $c_i = |\tau^{(i)}|+1$ if $\tau^{(i)} \in \{23\ldots m1: m \geq 2\}$. This composition will be one of the compositions in the set $C(\sigma).$
 Now, for every consecutive run of 1s in $\c(\sigma)$ of length $r$, the composition obtained by replacing that run of 1s by any of the $2^{r-1}$ compositions in $C(12\ldots r)$ will also be in the set $C(\sigma).$ Each run of 1s should be considered simultaneously. 

 \begin{example}
     Suppose $\sigma = 2341567(10)89(11)(12) =2341\oplus1\oplus 1\oplus 1 \oplus 312 \oplus 1 \oplus 1$. 
     Then $\c(\sigma) = (5,1,1,1,3,1,1)$ and \begin{align*}
        C(\sigma)  = & \, \{(5,1,1,1,3,1,1), (5,1,1,1,3,2), (5,2,1,3,1,1), (5,2,1,3,2), \\
        & \, (5,1,2,3,1,1), (5,1,2,3,2), (5,5,3,1,1), (5,5,3,2)\}.
     \end{align*}
 \end{example}

We are now ready to state the main theorem of this paper. 

\begin{theorem}
    For $n\geq 1$, the number of permutations in $\S_n$ that avoid the chain $(312,321: \sigma)$ for any $\sigma\in\Omega$ is equal to the number of composition of $n$ that avoid the set $C(\sigma).$ 
\end{theorem}

\begin{proof}
    Suppose that $\pi=\bigoplus_{i=1}^m \epsilon_{d_i}$ with $\pi^2=\bigoplus_{i=1}^m \epsilon^2_{d_i}$ for some $d_1+\cdots+d_m=n$ and that $\sigma = \bigoplus_{i=1}^k \tau^{(i)}$. We claim that $\pi$ avoids the chain $(312,321:\sigma)$ if and only if the composition $\d = (d_1,d_2,\ldots, d_m)$ avoids all compositions in $C(\sigma).$ 

    Since $\epsilon_d^2 = 34\ldots d12$, a single block of $\pi^2$ can only contain either one block of $\sigma$, or a consecutive sequence of blocks of size 1 (i.e., an increasing pattern). Additionally, $\epsilon_d^2$ must be length $d\geq k$ to contain $34\ldots k12$, length $d\geq k+1$ to contain $23\ldots k 1$, and length $d\geq k+2$ to contain $12\ldots k$ for $k\geq 3$ (and only length $d\geq k$ for $k\in\{1,2\}$). Note that $\epsilon_d^2$ contains 1 if $|\epsilon_d|\geq 1$,  contains 12 if $|\epsilon_d|\geq 2$, and contains $12\ldots r$ for $r\geq 3$ only if $|\epsilon_d|\geq r+2$, and so ``adding 2'' to any part greater than 2 of a composition of length  $k$ will determine exactly the lengths needed for $\pi^2$ to avoid $12\ldots k.$

    Now, let us consider the case when $\pi^2$ contains $\sigma.$ In this case, each block $\tau^{(i)}$ with $|\tau^{(i)}|>1$ must lie within a single block $\epsilon_{d_j}^2$ of $\pi^2$ and any increasing sequence of $\tau^{(i)}$s of size 1 must appear across blocks in $\pi^2.$ By construction of $C(\sigma),$ each $c_j$ is the minimum block length required to contain either a block $\tau^{(i)}$ or an increasing run of 1s, and thus the composition $\d$ will  dominate some $\c \in C(\sigma).$ 
    
    Similarly, if $\d$ contains some $\c\in C(\sigma),$ then by definition, there exist indices $i_1 < i_2 < \dots < i_r$ such that $d_{i_j} \geq c_j$ for each $j$. As observed above, by the construction of $C(\sigma)$, each $c_j$ is the minimum block length required to contain either a block $\tau^{(i)}$ or an increasing run of 1s. Because $d_{i_j} \geq c_j$, each block of the form $\epsilon_{d_{i_j}}^2$ is sufficiently long to contain the corresponding part of $\sigma$, and so $\pi^2$ contains $\sigma.$
\end{proof}

Let us note that the number of permutations that avoid $(312,321:\sigma)$ is equal to the number that avoid $(312,321:\sigma^{rci})$, where $\sigma^{rci}$ is the reverse-complement-inverse of $\sigma,$ as follows from \cite[Remark~4.1]{AG24}.

\renewcommand{\arraystretch}{2}
\begin{table}
    \centering
    \begin{tabular}{|c|c|c|c|}
    \hline
    $\sigma$ & $C(\sigma)$ &  $a_n(312,321:\sigma)$ & OEIS \\ \hline \hline
       123  & $(1,1,1), (1,2), (2,1), (5)$ &  0 (for $n\geq 5$)  & N/A\\ \hline
       132  &  (1,3) & \multirow{2}{*}{$F_{n+2}-1$} & \multirow{2}{*}{A000071} \\ \cline{1-2}
       213 & (3,1) & & \\ \hline
       231 & (4) & $T_{n+2}$ & A000073 \\ \hline
       312 & (3) & $F_{n+1}$ & A000045 \\ \hline \hline
       \multirow{2}{*}{1234} & $(1,1,1,1), (1,1,2),(1,2,1),$ &  \multirow{2}{*}{0 (for $n\geq 6$) } &  \multirow{2}{*}{N/A} \\ 
       &$(2,1,1),(2,2),(5,1), (1,5),(6)$ && \\ \hline
       1243 & $(1,1,3), (2,3)$ & \multirow{2}{*}{$2F_{n+1}-2$} & \multirow{2}{*}{A019274} \\ \cline{1-2}
       2134 & $(3,1,1), (3,2)$ & & \\ \hline
       1324 & (1,3,1) & $F_{n+3}-n-1$ & A000126 \\ \hline
       1342 & (1,4) & \multirow{2}{*}{$\displaystyle\sum_{i=1}^{n+1} T_i$}  & \multirow{2}{*}{A008937}\\ \cline{1-2}
       2314 & (4,1) & &\\ \hline
       1423 & (1,3) & \multirow{2}{*}{$F_{n+2}-1$}  & \multirow{2}{*}{A000071}\\ \cline{1-2}
       3124 & (3,1) & &\\ \hline
       2143 & (3,3) & $1+\frac{1}{5}(nF_{n+2} + (n-3)F_n)$ & $\text{A023610}+1$\\ \hline
       2341 & (5) & $Q_{n+3}$ & A000078 \\ \hline
       3412 & (4) & $T_{n+2}$ & A000073 \\ \hline
    \end{tabular}
    \caption{The number $a_n(312,321:\sigma)$ of permutations that avoid $312$ and $321$ whose square avoids $\sigma\in\Omega_k$ for $k\in\{3,4\}$ and $n\geq 2$ (unless otherwise stated), along with the associated composition for $\sigma.$ Here, $F_n$ denotes the $n$-th Fibonacci number, $T_n$ denotes the $n$-th Tribonacci number, and $Q_n$ denotes the $n$-th Tetranacci number.}
    \label{tab:321}
\end{table}

\begin{theorem}\label{thm:table}
    For $n\geq 2$, the values of $a_n(312,321:\sigma)$ in Table~\ref{tab:321} hold. 
\end{theorem}

\begin{proof}
    Let us consider the cases when $|\sigma|=3$ first. There are only five permutations in $\Omega_3$ to consider since $321\not\in\Omega_3$.
    For the pattern $\sigma=123$, we need to avoid \[C(123)=\{(1,1,1), (1,2), (2,1), (5)\}.\] It is clear that for $n\geq 5,$ it is impossible for a composition to avoid all of these patterns, so there are 0 permutations avoiding the chain $(312,321:123).$ 

    Next, consider those that avoid 132. In this case, $a_n(312,321:132)$ is equal to the number that avoid the composition $(1,3)$. By Theorem~\ref{thm:one comp}, there are $\sum_{i=1}^n b_{n-i}((3))$ such compositions. Since $b_n((3)) = F_{n+1},$ the $(n+1)$-st Fibonacci number, the result follows by a well-known identity about Fibonacci numbers (see OEIS A000071). Since $213$ is the reverse-complement-inverse of 132, the answer is the same for $a_n(312,321:213)$.
    For $\sigma=231,$ we need only count those compositions that avoid the composition $(4)$, which is $T_{n+2}$ and for $\sigma = 312$, we need only count those compositions that avoid $(3),$ which is $F_{n+1}.$

    Next let us consider those permutations in $\S_n(312,321:\sigma)$ with $|\sigma|=4.$ In this case, there are 11 cases to consider, but only 8 up to reverse-complement-inverse symmetry. For $\sigma=1234,$ we must enumerate compositions that avoid 
    \[
    C(1234) = \{(1,1,1,1), (1,1,2), (1,2,1), (2,1,1), (2,2), (1,5), (5,1), (6)\}.
    \] As before, it is clear this is impossible for $n\geq 6.$ In particular, any composition avoiding those compositions in $C(1234)$ must have at most 3 parts of size at most 5 since $(1,1,1,1), (6)\in C(1234)$, and only one part can be greater than 1 since $(2,2)\in C(1234)$. This excludes any compositions of $n\geq 8$, and the cases for $n\in\{6,7\}$ are straightforward to check directly. 
    
    Let us next consider $\sigma=1243.$ This amounts to enumerating compositions that avoid the compositions in $C(1243)=\{(1,1,3),(2,3)\}$. By Theorem~\ref{thm:mult comp}, the number of such compositions is given by \[b_{n-1}((1,3)) + \sum_{i=2}^n b_{n-i}(\{(1,3),(3)\}).\] We saw above that $b_{n-1}((1,3)) = F_{n+1}-1$ and that $b_{n-i}(\{(1,3),(3)\}) = b_{n-i}((3)) = F_{n-i+1}$, which has a sum of $F_{n+1}-1$, so in total there are $2F_{n+1}-2.$ Since $2134=1243^{rci},$ the result is the same for $2134.$ 

    Next, let's consider $\sigma=1324$, which has $C(1324)=\{(1,3,1)\}.$ The number of compositions that avoid $(1,3,1)$ satisfy
    \[
    b_n((1,3,1))=\sum_{i=1}^n b_{n-i}((3,1)),
    \] with $b_0((3,1))=1$ and  $b_n((3,1))=F_{n+2}-1$ for $n\geq 1$, as we found previously. By computing the sum, we get that $b_n((1,3,1))= F_{n+3}-n-1$.
    Now, let us consider the case when $\sigma=1342,$ corresponding to $C(1342)=\{(1,4)\}.$ Similar to the case when $\sigma=132$, we get that there are $\sum_{i=1}^n b_{n-i}((4))$ such compositions. Since $b_n((4))= T_{n+2}$, we get the sum $\sum_{i=1}^n T_{n-i+2}$, which can be reindexed to the sum in the table. Since $2314=1342^{rci}$, the same result holds for $\sigma=2314.$
    For $\sigma=1423$ and $\sigma=3124$, the result follows since $C(1423)=C(132)$ and $C(3124)=C(213).$

    Now, for the case when $\sigma=2143,$ we have $C(\sigma) = \{(3,3)\}$ and so 
    \[
    b_n((3,3))= b_{n-1}((3,3))+b_{n-2}((3,3)) + \sum_{i=3}^n b_{n-i}((3)),
    \]
    where $b_{n}((3)) = F_{n+1}.$ Thus if we let $b_n:=b_n((3,3)),$ we have $b_n=b_{n-1}+b_{n-2}+F_n-1.$ Solving this recurrence gives us the formula in the table.
    Finally, the cases when $\sigma=2341$ and $\sigma=3412$ follow immediately from the fact that $C(2341)=\{(5)\}$ and $C(3412)=\{(4)\}.$
\end{proof}

We conclude this section with a few more cases for when $\sigma$ takes a certain form and the answer is straightforward.

\begin{proposition}\label{prop:12-k}
    For $k\geq 3$ and $n> (k-1)^2$, we have $a_n(312,321:12\ldots k)=0.$
\end{proposition}
\begin{proof}
    Note that $C(12\ldots k)$ contains both the compositions $(1,1,\ldots, 1)$ of length $k$ and the composition $(k)$. However, this implies that a composition that avoids those compositions in $C(12\ldots k)$ must have at most $k-1$ parts of size at most $k-1$. If $n>(k-1)^2$, this is impossible and thus there are zero compositions avoiding $C(12\ldots k)$, implying there are zero permutations avoiding the chain $(312,321:12\ldots k).$
\end{proof}

The bound $n > (k-1)^2$ given in Proposition~\ref{prop:12-k} is sufficient to show that $a_n(312,321:12\dots k)$ is eventually zero, but it is not generally tight. For example, for $k=4$, the upper bound guarantees $a_n = 0$ for $n > 9$, but as stated in Table~\ref{tab:321}, we have that $a_n = 0$ for $n \geq 6$.

\begin{proposition}
    For $n\geq 1$ and $k\geq 3$ ,
    \[a_n(312,321:\sigma)=\begin{cases} F^{(k)}_{n+k-1} & \sigma=23\ldots k1 \\[4pt]
     F^{(k-1)}_{n+k-2} & \sigma=34\ldots k12,
    \end{cases}\] where $F^{(k)}_n$ denotes the $n$-th $k$-nacci number. 
\end{proposition}

\begin{proof}
    In the case $\sigma=23\ldots k1$, we have $C(\sigma)=\{(k+1)\}$ and in the case that $\sigma = 34\ldots k12$, we have that $C(\sigma)=\{(k)\}$, and thus the result clearly follows (see Example~\ref{ex:Fibonacci}).
\end{proof}

Finally, let us note that it is clear from the previous section that to enumerate permutations avoiding the chain $(312,321:S)$ for any set $S$ of permutations, it is enough to enumerate those compositions that avoid $\bigcup_{\sigma\in S} C(\sigma)$. 



\section{Permutations avoiding the chain $(312,4321:\sigma)$ for $\sigma\in\S_3$}\label{sec:312-4321}

In the case that $\pi$ avoids both $312$ and $4321,$ the form $\pi$ takes is not quite as nice as the case in the previous section. However, we can still write $\pi$ as a direct sum. In particular $\pi = \rho\oplus \tau$ where $\rho\in\S_k(312,4321)$ ends in 1 and $\tau \in \S_{n-k}(312,4321)$ for some $1\leq k \leq n$. Note that $\tau$ could be the empty permutation, but $\sigma$ has length $k\geq 1.$
Table~\ref{tab:4321} summarizes the results of this section.

\renewcommand{\arraystretch}{2}
\begin{table}[h!]
    \centering
    \begin{tabular}{|c|c|c|c|}
    \hline
    $\sigma$ &  $a_n(312,4321:\sigma)$ & OEIS & Theorem \\ \hline \hline
       123  &  eventually 0  & N/A & Theorem~\ref{thm:123} \\ \hline
       132   & \multirow{2}{*}{$T_{n+2}+T_{n+1}-1$} & \multirow{2}{*}{$\text{A001590}-1$} & Theorem~\ref{thm:132}  \\ \cline{1-1}\cline{4-4}
       213 & & & Corollary~\ref{cor:213} \\ \hline
       231 & \phantom{$\dfrac{1}{\mid}$}  g.f. $A(x) = \dfrac{1}{1-x-x^2-2x^3-3x^4-2x^5}$ \phantom{$\dfrac{1}{\mid}$}&  A381858& Theorem~\ref{thm:231}  \\ \hline
       312 & $Q_{n+3}$ & A000078& Theorem~\ref{thm:312}  \\ \hline 
       321 &\phantom{$\dfrac{1}{\mid}$} g.f. $B(x)=\dfrac{1-x}{1-2x-x^3+x^5}$\phantom{$\dfrac{1}{\mid}$} &  A381859 & Theorem~\ref{thm:321} \\ \hline 
    \end{tabular}
    \caption{The number of permutations that avoid 312 and 4321 whose square avoids $\sigma\in\S_3$. Here, $T_n$ is the $n$-th Tribonacci number and $Q_n$ is the $n$-th Tetranacci number.}
    \label{tab:4321}
\end{table}

    \begin{theorem}\label{thm:123}
        For $n\geq 7,$ $a_n(312,4321:123)=0$.
    \end{theorem}
    \begin{proof}
        Let us first note that if $\pi =\rho\oplus\tau$ with $\rho$ ending in 1, then if $|\tau|\geq 1$, $\tau$ must also end with 1. Otherwise, we could write $\tau=\tau_1\oplus\tau_2$ with $|\tau_2|\geq1$ and thus $\pi^2 = \rho^2\oplus\tau_1^2\oplus\tau_2^2$ which necessarily contains a 123 pattern. We will show that there are no permutations in $\S_n(312,4321:123)$ that end with 1 for $n\geq 8.$ Together with the note above, this would imply there are no permutations in $\S_n(312,4321:123)$ for $n\geq 15.$ The cases for $7\leq n\leq 14$ are easily verified by computer using built-in Sage packages \cite{Sage}.

        For the sake of contradiction, suppose $\pi =\pi_1\ldots \pi_{n-1}1 \in \S_n(312,4321:123)$ with $n\geq 8.$ Then since $\pi$ avoids 312 and 4321, we must have that $\pi_1\ldots\pi_{n-1}$ avoids 312 and 321 (and is thus of the form described in the previous section, with values shifted up by 1). This implies that $\pi_1\pi_2\pi_3$ takes one of the forms:
        \[
        \{234,235,243,245,324,325,342,345\}.        \]
        Let us proceed by cases. If $\pi_1\pi_2\pi_3=234,$ then $34n$ is a 123 pattern in $\pi^2.$ If $\pi_1\pi_2\pi_3=235,$ then $35n$ is a 123 pattern in $\pi^2.$ If $\pi_1\pi_2\pi_3=243,$ then $\pi_4\in\{5,6\}$, so either $45n$ or $46n$ is a 123 pattern in $\pi^2.$ If $\pi_1\pi_2\pi_3=245,$ then $\pi_4\in\{3,6\}$. If $\pi_1\pi_2\pi_3\pi_4=2456$, then $46n$ is a 123 pattern in $\pi^2$. If $\pi_1\pi_2\pi_3\pi_4=2453$, then $\pi_5\in\{6,7\}$ and so $46n$ or $47n$ is a 123 pattern in $\pi^2.$ If $\pi_1\pi_2\pi_3=324,$ then $\pi_4\in\{5,6\}$ so either $45n$ or $46n$ is a 123 pattern in $\pi^2.$ If $\pi_1\pi_2\pi_3=325$, then $\pi_4\in\{4,6\}.$ If $\pi_4=4$, then $\pi_5\in\{6,7\}$, so $24n$ is a 123 pattern in $\pi^2.$ If $\pi_4=6$ then $\pi_5\in\{4,7\}$ so either $24n$ or $27n$ is a 123 pattern in $\pi^2.$ If $\pi_1\pi_2\pi_3=342$, then $24n$ is a 123 pattern in $\pi^2.$
        Finally, if $\pi_1\pi_2\pi_3=345$, then $\pi_4\in \{2,6\}.$ If $\pi_4=6,$ then $56n$ is a 123 pattern. If $\pi_4=2,$ then $\pi_5\in\{6,7\}$ so either $56n$ or $57n$ is a 123 pattern in $\pi^2.$ Therefore, we have no such permutations in $\S_n(312,4321:123)$ for $n\geq 8.$
    \end{proof}

\begin{theorem}\label{thm:132}
    For $n\geq 1$, $a_n(312,4321:132) = T_{n+2}+T_{n+1}-1$, where $T_n$ is the $n$-th Tribonacci number. 
\end{theorem}
\begin{proof}
    The initial conditions are easily checked, so we let $n> 5.$ In this case, we can write $\pi = \rho\oplus\tau$ where $\rho\in\S_k(312,4321)$ ends in 1 and $\tau\in\S_{n-k}(312,4321)$. Note that $\pi^2 = \rho^2\oplus\tau^2$, so if $\pi \in \S_n(312,4321:132)$, then we must have that $\rho\in\S_k(312,4321:132)$ and 1 appears in the first $k$ elements of $\pi^2$. To avoid 132, this means that $\tau^2$ is the identity permutation, i.e., is an involution. It is well known that 312-avoiding involutions are of the form $\bigoplus_i \delta_{d_i}$ for some composition $\d=(d_1,\ldots,d_m)$, where $\delta_d=d(d-1)\ldots 21.$ However, since $\pi$ also avoids 4321, the corresponding composition $\d$ is composed only of 1s, 2s, and 3s, and are thus enumerated by the Tribonacci numbers $T_{n+2}.$

    Therefore, we need only know how many permutations in $\S_n(312, 4321: 132)$ end in 1. 
    Let $\pi = \pi_1\ldots \pi_{n-1}1$ be such a permutation. Note that since $\pi$ avoids 312 and 4321, $\pi_1\in\{2,3\},$ since otherwise either $\pi_1321$ or $\pi_123$ would be a subsequence and thus would give us either a 4321 or 312 pattern. If $\pi_1=2,$ then for a similar reason $\pi_2\in\{3,4\}$.

    If $\pi_1\pi_2=23,$ then since $\pi_n=1$, we have $\pi^2_n=2$. Since $\pi^2$ avoids 132, we must have $\pi^2_{n-1}=1$ and since $\pi_1\pi_2=23$, we have $\pi_1^2=3.$ Thus $\pi^2 = 3\pi_2^2\ldots \pi_{n-1}^212.$ Since $\pi^2$ avoids 132, it must be that $\pi_2^2\ldots \pi_{n-1}^2$ is increasing and thus $\pi = 234\ldots n1.$

    If  $\pi_1\pi_2=24,$ then similarly to above we have $\pi^2 = 4\pi_2^2\ldots \pi_{n-1}^212.$ If $\pi_j=3$, then all elements in $\pi_3\ldots\pi_{j-1}$ are exactly $\{5,6,\ldots, j+1\}$. However, if $j\geq 4$, this implies that $\pi^2_1\pi^2_{j-1}\pi^2_{j} = 4\pi_{j+1}5$ where $\pi_{j+1}\geq j+2\geq 6$, which is a 132 pattern.
    Thus we have $\pi_3=3,$ so $\pi^2 =4\pi_2^23\ldots \pi_{n-1}^212$ and thus $\pi_2^2\pi^2_4\ldots \pi_{n-1}^2$ is increasing, so $\pi = 243567\ldots n1.$

    Finally, if $\pi_1=3$, then $\pi_2\in\{2,4\}$. If $\pi_1\pi_2=32$, then $\pi^2 = \pi_1^22\ldots \pi_{n-2}^23$ and so $2 m 3$ is a 132 pattern for some $m \in \{\pi_{3}^2,\ldots, \pi_{n-2}^2\}$ since $n>5.$ If $\pi_1\pi_2=34$ and $\pi_k=2$, then in order to avoid 132, $k\in\{n-2,n-1\}$. However, since the elements that appear before 2 must be smaller than the elements that appear after 2 (except $\pi_n=1$), and the elements that appear before 2 must be in increasing order in order for $\pi$ to avoid 4321, we must have that  $\pi_{k-2}=k$, giving us a contradiction. 

    Therefore the only two permutations in $\S_n(312,4321:132)$ with $\pi_n=1$ and $n>5$ are $234\ldots n1$ and $2435\ldots n1.$ For $n\leq 5$, the only such permutations are $1, 21, 231, 321, 2341,3241, 23451,$ and $24351.$ Therefore, by considering the position  of 1 in $\pi$, we can write 
    \begin{align*}a_n(312,4321:132) &=  T_{n+1} + T_{n} + 2T_{n-1} + 2 T_{n-2} + \cdots +2T_1\\
    &=T_{n+1}+T_{n} + 2\sum_{i=1}^{n-1} T_i \\
    &=T_{n+2}+T_{n+1}-1
    \end{align*}
    where the last equality follows from a well-known identity about the sum of the Tribonacci numbers (see for example OEIS A008937).
    \end{proof}

As a corollary, we obtain $a_n(312,4321:213)$ as well since the reverse-complement-inverse of the set $\{312,4321\}$ is $\{312,4321\}$ and the reverse-complement-inverse of $132$ is 213.
\begin{corollary}\label{cor:213}
    For $n\geq 1$, $a_n(312,4321:213) = T_{n+2}+T_{n+1}-1$.
\end{corollary}

 \begin{theorem}\label{thm:231}
        The generating function $A(x) = \sum_{n\geq 0}a_n(312,4321:231)x^n$ is given by
        \[
        A(x) = \frac{1}{1-x-x^2-2x^3-3x^4-2x^5}.
        \]
    \end{theorem}
    \begin{proof}
        The initial conditions are easily checked. As before, we can write $\pi = \rho\oplus\tau,$ where $\rho$ ends in 1. Since any occurrence of $231$ in $\pi^2$ must occur completely within $\rho^2$ or $\tau^2$, we have that 
        \[
        a_n(312,4321:231) = \sum_{i=1}^n b_i(312,4321:231)a_{n-i}(312,4321:231),
        \]
         where $b_n(312,4321:231)$ is the number of permutations in $\S_n(312,4321:231)$ that end with 1. Thus it suffices to find $b_n(312,4321:231).$

        Let $\pi = \pi_1\ldots\pi_{n-1}1\in \S_n(312,4321:231)$ for some $n\geq 6$. We know that $\pi_1\in\{2,3\}.$ First consider the case that $\pi_1=2.$ Then we must have $\pi_2\in\{3,4\}$. Note that $\pi^2 = \pi_2\pi^2_2\ldots \pi_{n-1}^2 2$. Therefore $\pi_2 n 2$ is a 231 pattern in $\pi^2$. If $\pi_1=3,$ then $\pi_2\in\{2,4\}$ and $\pi_3\in\{2,4,5\}.$ Thus, in $\pi^2 = \pi_3\pi_2^2\ldots\pi_{n-1}^2 3$ either $\pi_3n3$ is a 231 pattern or $\pi_3 =2.$ But then $\pi_2=4$ and so $\pi^2_3=4$, so $4 n 3$ or $4(n-1)3$ is a 231 pattern in $\pi^2.$ Thus if $n\geq 6,$ there are no permutations in $\S_n(312,4321:231)$ that end in 1. 

        If $n<6$, it is easily verified that the permutations avoiding the chain $(312, 4321:231)$ that end in 1 are those in the set 
        \[
        \{1,21,231,321,2431,3241,3421,32541,34521\}.
        \] This implies that $a_n := a_{n}(312,4321:231)$ satisfies:
        \[
        a_n = a_{n-1} + a_{n-2} + 2a_{n-3} + 3a_{n-4} + 2a_{n-5},
        \] which together with initial conditions, implies $a_n(312,4321:231)$ has the associated generating function given in the statement of the theorem.
    \end{proof}

    Note that in the next theorem, the result given can also be described as the number of permutations that strongly avoid 312 and classically avoid 4321. Recall that in \cite{PG24}, the authors do enumerate those permutations that strongly avoid 312 and avoid some patterns $\sigma\in\S_4$. However, they do not address this particular case in that paper.
    
    \begin{theorem}\label{thm:312}
        For $n\geq 1$, $a_n(312,4321:312) = Q_{n+3}$, where $Q_n$ is the $n$-th Tetranacci number.
    \end{theorem}
    \begin{proof}
        The initial conditions are easily checked. As before, we can write $\pi = \rho\oplus\tau,$ where $\rho$ ends in 1. Since any occurrence of $312$ in $\pi^2$ must occur completely within $\rho^2$ or $\tau^2$, we have that 
        \[
        a_n(312,4321:312) = \sum_{i=1}^n b_i(312,4321:312)a_{n-i}(312,4321:312),
        \]
        where $b_n(312,4321:312)$ is the number of permutations in $\S_n(312,4321:312)$ that end with 1. Thus it suffices to find $b_n(312,4321:312).$

        In \cite[Theorem 3.1]{BS19}, B\'{o}na and Smith characterize those permutations in $\S_n(312:312)$ that end in 1. By that result, the only permutations in $\S_n(312,4321:312)$ that end in 1 are those in the set
        \[
        \{1,21,321,3421\}.
        \]
        By considering the position of 1 in $\pi$, it follows that \[a_n(312,4321:312) = \sum_{i=1}^4a_{n-i}(312,4321:312),\] which is exactly the Tetranacci recurrence.
    \end{proof}

    \begin{theorem}\label{thm:321}
         The generating function $B(x) = \sum_{n\geq 0}a_n(312,4321:321)x^n$ is given by
        \[
        B(x) = \frac{1-x}{1-2x-x^3+x^5}.
        \]
    \end{theorem}
    \begin{proof}
        As before, we can write $\pi = \rho\oplus\tau,$ where $\rho$ ends in 1. Since any occurrence of $321$ in $\pi^2$ must occur completely within $\rho^2$ or $\tau^2$, and thus it suffices to find $b_n(312,4321:321)$, the number of such permutations that end in 1.

        As before, if $\pi_n=1$ and $n\geq 5$, then $\pi_1\in\{2,3\}$. If $\pi_1=2,$ then we claim that $\pi = 234\ldots n1.$ If not then there is some smallest $k$ so that \[\pi = 23\ldots k(k+2)(k+3)\ldots (k+i)(k+1)\ldots 1\]
        but then in $\pi^2,$ $(k+2)(k+1)2$ is a 321 pattern. 
        Now consider the case when $\pi_1=3$. If $\pi_2=2$, then $\pi_1^2\pi_2^2 = \pi_3\pi_2 = \pi_32$ with $\pi_3\neq 1$ since $n\geq 5.$ Thus $\pi_321$ is a 321 pattern. If $\pi_j=2$ for $j>3$, then $\pi_2$ is 4 and $\pi_3>4$. Thus $\pi_1^2\pi_j^2\pi_n^2=\pi_343$ is a 321 pattern. Finally, if $\pi_3=2,$ then $\pi_1\pi_2\pi_3=342$, so $\pi^2=2\pi_44\ldots 3,$ where $\pi_4>3$ since $n\geq 5$. Thus $\pi_4 4 3$ is a 321 pattern in $\pi^2.$

        This all implies that there is only one permutation in $\S_n(312,4321:321)$ that ends in 1 when $n\geq 5.$ For $n<5$, it is straightforward to check that the permutations that end in 1 are those in the set:
        \[
        \{1,21,231,321,2341,3421\}.
        \]
        Thus it follows that $a_n(312,4321:321)$ satisfies
        \[
        a_n(312,4321:321) = a_{n-3}(312,4321:321)+a_{n-4}(312,4321:321)+\sum_{i=0}^{n-1} a_i(312,4321:321).
        \]
        This implies the associated generating function $B(x) = \sum_{n\geq0} a_n(312,4321:321)x^n$ satisfies \[B(x) = 1+ (x^3+x^4) B(x) + \frac{xB(x)}{1-x},\] which in turn, implies the statement of the theorem.
    \end{proof}

\section{Further directions and open questions}

There are many future possible directions of research. One interesting question that could be addressed by using some of the ideas from this paper is to enumerate permutations $\pi\in\S_n$ that avoid 312 and 321 whose \textit{cube} (or a higher power) avoids a pattern $\sigma.$ It seems like you may be able to answer this question generally in terms of the related compositions, but not directly in terms of composition avoidance as was done in this paper. For example, based on the ideas similar to the ones in this paper, we conjecture that the number of permutations that avoid the chain $(312,321:\varnothing:2143)$ (i.e, those with the property that $\pi$ avoids $\{312,321\}$ and $\pi^3$ avoids $2143$) is equal to the number of compositions $\d=(d_1,d_2,\ldots,d_m)$ of $n$ so that all $d_i$ are 1 or 3, except for at most one. Section 5 of \cite{AG24} addresses a similar question for when $\sigma\in\S_3$. Those ideas combined with the ones in this paper may allow one to enumerate permutations avoiding any chain of any length of the form $(312,321:\tau_1:\tau_2:\cdots)$ for any patterns $\tau_i\in\bigcup_{m\geq 1} \S_m$.

Another possible direction is  to enumerate permutations that avoid the chain $(312,k\ldots 21:\sigma)$ for some $\sigma\in\S_3$ and some $k\geq 5.$ For example, we conjecture that if we take $a_n:=a_n(312,54321: 132)$ then \[a_n= a_{n-1}+a_{n-2}+a_{n-3}+a_{n-4}+n-1\] for $n\geq 6.$ Perhaps if this can be extended for further $k$, one could use these ideas to enumerate $a_n(312:132)$ itself, or $a_n(312:\sigma)$ for some other $\sigma\in\S_3\setminus\{312\}$, each of which is currently unknown, though lower bounds for these numbers can be found in \cite{AG24}.

\nocite{*}
\bibliographystyle{plainnat}
\bibliography{refs}
\label{sec:biblio}

\subsection*{Statements and Declarations}
The authors declare that no funds, grants, or other support were received during the preparation of this manuscript. 
The authors have no relevant financial or non-financial interests to disclose.
The views expressed in this article do not necessarily represent 
the views or opinions of the U.S. Naval Academy, Department of the Navy, or Department of Defense or any of its components.

\subsection*{Data availability}
 In this manuscript, there is no experimental or computer calculated data, and there is only logical proof.
\end{document}